\documentclass[letterpaper, 10 pt, conference]{ieeeconf}  % Comment this line out if you need a4paper

\IEEEoverridecommandlockouts                              % This command is only needed if 
\usepackage{common_pkgs}

\title{\LARGE \bf
Exploiting Over-Approximation Errors as Preview Information for Nonlinear Control*
}

\author{Antoine Aspeel$^{1}$, Antoine Girard$^{1}$, and Thiago Alves Lima$^{2}$% <-this % stops a space
\thanks{*
A.G. received support from the RTE-CentraleSup\'elec chair. T.A.L. was supported by CNPq (Conselho Nacional de Desenvolvimento Cient\'ifico e Tecnol\'{o}gico) via the grant number 443674/2024-8.}% <-this % stops a space
\thanks{$^{1}$Universit{\'e} Paris-Saclay, CNRS, CentraleSup{\'e}lec, Laboratoire des signaux et syst{\`e}mes, 91190, Gif-sur-Yvette, France.
{\tt\small firstname.lastname@centralesupelec.fr}}%
\thanks{$^{2}$Systems Engineering Division, Aeronautics Institute of Technology (ITA), Fortaleza, 60.415-513, CE, Brazil.
{\tt\small thiago.lima@gp.ita.br}}%
}
\begin{document}

\maketitle
\thispagestyle{empty}
\pagestyle{empty}

%%%%%%%%%%%%%%%%%%%%%%%%%%%%%%%%%%%%%%%%%%%%%%%%%%%%%%%%%%%%%%%%%%%%%%%%%%%%%%%%

\begin{abstract}
We study the control of nonlinear constrained systems via over-approximations. Our key observation is that the over-approximation error, rather than being an unknown disturbance, can be exploited as input-dependent preview information. This leads to the notion of informed policies, which depend on both the state and the error. We formulate the concretization problem---recovering a valid input for the true system from a preview-based policy---as a fixed-point equation. Existence of solutions follows from the Brouwer fixed-point theorem, while efficient computation is enabled through closed-form, linear, or convex programs for input-affine systems, and through an iterative method based on the Banach fixed-point theorem for nonlinear systems.
\end{abstract}

%%%%%%%%%%%%%%%%%%%%%%%%%%%%%%%%%%%%%%%%%%%%%%%%%%%%%%%%%%%%%%%%%%%%%%%%%%%%%%%%
\section{INTRODUCTION}

The control of nonlinear systems subject to state and input constraints is a fundamental challenge in systems and control. Constraints arise naturally in practice from actuator limits in aerospace systems, to safety envelopes in autonomous driving and robotics, to operating ranges in energy systems. Designing controllers that enforce such constraints while achieving the desired behavior is central to the deployment of control systems in safety-critical domains.

While linear models and unconstrained settings have led to elegant and powerful theory, real-world systems are often nonlinear and must operate within hard bounds. The main challenges lie in ensuring constraint satisfaction without excessive conservatism, and in achieving scalability to systems of realistic complexity. These difficulties motivate the development of new methods that can provide tractable and reliable control laws for constrained nonlinear dynamics.

\paragraph{Related work}
Abstraction methods replace a nonlinear system by an abstract one that is easier to analyze and control. Two main cases can be distinguished. The first is symbolic control, where a finite-state abstraction is obtained through discretization of the state and input spaces, enabling controller synthesis with formal guarantees using automata- or game-theoretic techniques \cite{reissig2016feedback, belta2017formal, liu2014abstraction}. The second is based on over-approximations of the dynamics, typically by linear or piecewise affine systems (also called hybridizations) \cite{althoff2008reachability, girard2011synthesis}. These models make it possible to apply tractable synthesis and verification methods while ensuring safety through conservative error bounds. Moreover, some recent approaches combine symbolic control and over-approximation techniques to benefit from the advantages of both frameworks \cite{calbert2024smart}. In all these cases, the approximation error is treated as an unknown adversarial disturbance, which may introduce significant conservatism. In contrast, the present work leverages the structure of the over-approximation error by interpreting it as input-dependent preview information that can be exploited by the control policy.

Control with preview information refers to control strategies that exploit knowledge of future exogenous signals, such as reference trajectories or disturbances, to improve closed-loop performance (see \cite{birla2015optimal} and references therein). Preview has also been used in safety-critical control problems, where knowledge of future disturbances allows to enlarge the robust controlled invariant set \cite{liu2024quantifying}. The key difference with the present work is that, in existing formulations, the preview signal is exogenous and independent of the control input, whereas here it arises endogenously from the over-approximation error and depends explicitly on the chosen input.

\paragraph{Contributions}

We show that over-approximations naturally provide \emph{input-dependent preview information}. This motivates policies that depend not only on the state but also on the over-approximation error, in contrast to classical approaches that treat the error as an adversarial disturbance.

We formalize the \emph{concretization problem} ---recovering a valid input for the true system from a preview-based policy--- as a fixed-point equation that captures the mutual dependence between the input and the over-approximation error. Using the Brouwer fixed-point theorem, we establish a general existence result.

For input-affine systems, we show that concretization can be performed efficiently: in closed form when the true and approximate dynamics share the same input matrix, or via a convex (or even linear) program when the input matrix differs.

For nonlinear systems, we establish contraction conditions that allow effective iterative concretization via the Banach fixed-point theorem.

\paragraph{Notation}
For two vectors $x$ and $y$, their vertical concatenation is denoted $(x,y)$, and we use the notation $x_{i:j}\coloneqq (x_i,\dots, x_j)$. For a matrix $M$, $M_{i,:}$ is its $i$-th row, and $M_{i:j,:}$ is the submatrix of $M$ formed by the rows $\{k\mid i\leq k\leq j\}$ (similarly for columns). $\|\cdot\|$ denotes a vector norm when applied to a vector, and the corresponding induced norm when applied to a matrix. For a function $f$, its Lipschitz constant is denoted $L_f$.

\section{OVER-APPROXIMATIONS}

We are interested in known nonlinear and deterministic dynamical systems under state and input constraints, i.e.,
\begin{equation}\label{eq:nonlinearSystem}
x_{t+1}=f(x_t,u_t),\ \ \ (x_t,u_t)\in\X\times\U,
\end{equation}
where $f:\X\times\U\rightarrow\R^{n_x}$, $\X\subseteq\R^{n_x}$ and $\U\subseteq\R^{n_u}$.

Since the dynamics $f$ is nonlinear, such a system is hard to control. For that reason, it will be approximated by a simpler (e.g., linear) dynamics $\hat{f}:\X\times\U\rightarrow\R^{n_x}$ for which synthesizing a policy is easier. However, such an approximation introduces an error $e$, namely:
\begin{align}\label{eq:error}
e(x,u)=f(x,u)-\hat{f}(x,u).
\end{align}
Let $\mathcal{E}\subseteq\R^{n_x}$ be a set that over-approximates the error over the domain, i.e.,
\begin{align}\label{eq:errorInclusion}
\text{ for all } (x,u)\in\X\times\U:\ e(x,u)\in\E,
\end{align}
and consider the dynamics $F:\X\times\U\times\E\rightarrow\R^{n_x}$ defined by
$$
F(x,u,\bar{e})=\hat{f}(x,u)+\bar{e},
$$
where $\bar{e}\in\E$ is a disturbance. The function $F$ is called an \emph{over-approximation of $f$} because it satisfies
$$
f(x,u)\in\{F(x,u,\bar{e})\mid\bar{e}\in\E\}
$$
for all $(x,u)\in\X\times\U$.

The disturbance $\bar{e}$ is usually treated as unknown \cite{althoff2008reachability, girard2011synthesis}, and a state-feedback policy $\pi:\X\rightarrow\U$ is designed to complete a control task for any realization of the disturbances. In contrast, in the next section, we show why the disturbance can be considered as known at runtime, and how it can be exploited by the policy.

\section{INFORMED POLICIES}\label{sec:IP}

By construction, the error incurred in state $x$ under input $u$ is $\bar{e}=e(x,u)$. At runtime, $x$ is known and $u$ is a decision variable. Since the error function $e$ is known, $\bar{e}$ is a deterministic quantity that depends solely on $u$. Thus, $\bar{e}$ can be interpreted as \emph{input-dependent preview information}. Consequently, rather than synthesizing a feedback policy that depends only on the state, one can design a policy that depends jointly on the state and the error, i.e.,
$$
\pi:\X\times\E\rightarrow\U.
$$
We call such a policy an \emph{informed policy}, and a policy whose domain is $\X$ will be called \emph{uninformed}.

The use of a preview-based policy leads to an interplay between the policy $\pi$ and the over-approximation error $e$ as described in Fig.~\ref{fig:preview_diagram}. It follows that at runtime, once $x\in\X$ is available, the input $u$ must be chosen to satisfy $u=\pi(x,e(x,u))$. More explicitly, concretizing an informed policy requires to solve the following problem:
\begin{problem}[Concretization]\label{prob:concretization}
Find $u\in\U$ such that
$$
u=\pi(x,f(x,u)-\hat{f}(x,u)).
$$
\end{problem}
If the policy $\pi(x,e)$ is constant with respect to its second argument, i.e., it does not depend on the error, then Problem~\ref{prob:concretization} necessarily has a solution. This shows that informed policies generalize uninformed policies.

\begin{figure}
\centering
\begin{tikzpicture}[scale=0.5, font=\small, >=Stealth]
  % Define constant
  \def\yshift{15}

  % Define Coordinates
  \coordinate (A) at (-6,0);
  \coordinate (B) at (6,0);

  % Draw Nodes for A and B
  \node (NodeA) at (A) [left, align=center] {Input set \\ $\U$};
  \node (NodeB) at (B) [right, align=center] {Error set \\ $\E$};

  % Path from A to B (Curved Upwards)
  % "bend left" curves the line towards the left relative to the direction of travel
  \draw[->, thick] ([yshift=\yshift]A) to[bend left=10] 
    node[midway, above, align=center] {Disturbance as an input-dependent preview information \\ $\bar{e}_t=f(x_t,u_t)-\hat{f}(x_t,u_t)$} ([yshift=\yshift]B);

  % Path from B to A (Curved Downwards)
  \draw[->, thick] ([yshift=-\yshift]B) to[bend left=10] 
    node[midway, below, align=center] {Input from a preview-based policy \\ $u_t=\pi(x_t,\bar{e}_t)$} ([yshift=-\yshift]A);
\end{tikzpicture}
\caption{Illustration of the interplay between the informed policy and the approximation error.}
\label{fig:preview_diagram}
\end{figure}

The next theorem states that informed policies concretized by solving Problem~\ref{prob:concretization} lead to over-approximations of $f$.
\begin{theorem} \label{thm:overapproximation}
Let $f,\hat{f}:\X\times\U\rightarrow\R^{n_x}$ and $\pi:\X\times\E\rightarrow\U$ be three functions, and assume that $\E$ satisfies \eqref{eq:errorInclusion}. For any $x\in\X$, if Problem~\ref{prob:concretization} has a solution $u$, then
$$
f(x,u)\in \{ F(x,\pi(x,\bar{e}),\bar{e}) \mid \bar{e}\in\E \}.
$$
\end{theorem}
\begin{proof}
Consider $\bar{e} \coloneq f(x,u)-\hat{f}(x,u)$ and note that $\bar{e}\in\E$ thanks to \eqref{eq:errorInclusion}. In addition, since $u$ is a solution to Problem~\ref{prob:concretization}, $u=\pi(x,\bar{e})$. Finally, one can write
\begin{align*}
f(x,u) &= \hat{f}(x,u)+f(x,u)-\hat{f}(x,u) \\
&=\hat{f}(x,\pi(x,\bar{e}))+\bar{e} \\
&= F(x,\pi(x,\bar{e}),\bar{e}),
\end{align*}
which concludes the proof.
\end{proof}

To concretize an informed policy $\pi$ at runtime, it must be chosen such that (i) Problem~\ref{prob:concretization} has a solution, and (ii) a solution can be found efficiently. Overall, solving a control task for $f$ can be done through the following steps:
\begin{enumerate}[1.]
\item Find $\hat{f}$ and $\E$ such that $F$ is an over-approximation of $f$, i.e., inclusion \eqref{eq:errorInclusion} holds. \label{item:overapprox}
\item Find a set $\Pi$ of informed policies $\pi:\X\times\E\rightarrow\U$ such that for any $x\in\X$, Problem~\ref{prob:concretization} has a solution that can be found efficiently at runtime.

\item Find an informed policy $\pi\in\Pi$ such that the closed-loop system
$$
x_{t+1}=F(x_t,\pi(x_t,\bar{e}_t),\bar{e}_t)
$$
accomplishes the control task of interest for any $\bar{e}_t\in\E$.
\item At runtime, once the state $x\in\X$ is available, find a control input by solving Problem~\ref{prob:concretization}.
\end{enumerate}
Computing over-approximations (step~1.) and synthesizing informed policies (step~3.) are two well-studied problems discussed in Sections~\ref{sec:computingOA} and \ref{sec:synthesizingIP}, respectively. In the remainder of this section, we investigate the existence of solutions to Problem~\ref{prob:concretization} and derive sufficient conditions that enable it to be solved efficiently (step~2.).

\subsection{Concretization as a fixed point}\label{sec:IPconcretization}
For a function $g:X\rightarrow X$, a \emph{fixed point of $g$} is a point $x\in X$ such that $x=g(x)$. Such an equation is referred to as a \emph{fixed-point equation}. We observe that for any $x\in\X$, an input $u$ is a solution to Problem~\ref{prob:concretization} if and only if it is a fixed point of the function $\mathcal{F}_x:\U\rightarrow\U$ defined by
\begin{align}\label{eq:FPoperator}
u\mapsto \pi(x,f(x,u)-\hat{f}(x,u)).
\end{align}
This observation allows us to exploit several results from fixed-point theory to study Problem~\ref{prob:concretization}.

First, let us recall Brouwer fixed-point theorem \cite{brouwer1911abbildung} which establishes the existence of a fixed point under mild conditions (see, e.g., \cite[Corollary~6.6]{border1985fixed} for a comprehensive presentation).

\begin{lemma}[Brouwer fixed-point theorem]\label{thm:brouwer}
Let $X\subseteq\R^n$ be a non-empty, compact and convex set, and let $g:X\rightarrow X$ be a function. If $g$ is continuous, then it admits a fixed point.
\end{lemma}

The following theorem is a direct application of Brouwer fixed-point theorem to the function $\mathcal{F}_x$ defined by \eqref{eq:FPoperator}. It establishes mild sufficient conditions for the existence of a solution to Problem~\ref{prob:concretization}.

\begin{theorem}\label{thm:existence}
Let $f,\hat{f}:\X\times\U\rightarrow\R^{n_x}$ and $\pi:\X\times\E\rightarrow\U$ be three functions continuous with respect to their second argument, and assume that $\E$ satisfies \eqref{eq:errorInclusion}. If the set $\U$ is non-empty, compact and convex, then Problem~\ref{prob:concretization} has a solution for all $x\in\X$.
\end{theorem}
\begin{proof}
It follows from the continuity assumptions over $f$, $\hat{f}$ and $\pi$ that the function $\mathcal{F}_x:\U\rightarrow\U$ defined by \eqref{eq:FPoperator} is continuous for all $x\in\X$. Since $\U$ is non-empty, compact and convex, it follows from Lemma~\ref{thm:brouwer} that for all $x\in\X$, $\mathcal{F}_x$ has a fixed point $u\in\U$. This fixed point satisfies $u=\pi(x,f(x,u)-\hat{f}(x,u))$ and is therefore a solution to Problem~\ref{prob:concretization}.
\end{proof}

\begin{remark}\label{rem:existence}
If the set $\U$ does not satisfy the convexity or compactness assumption, but that $\E$ does, the conclusion of Theorem~\ref{thm:existence} remains. Indeed, Brouwer fixed-point theorem can be applied to the following fixed-point equation in the variable $\bar{e}\in \E$:
$$
\bar{e}=f(x,\pi(x,\bar{e}))-\hat{f}(x,\pi(x,\bar{e})),
$$
where every fixed point $\bar{e}$ yields a solution $u\coloneqq\pi(x,\bar{e})$ to Problem~\ref{prob:concretization}. Note that the set $\E$ is chosen when constructing an over-approximation, and can always be taken convex.

Finally, the convexity of the set $\U$ (or $\E$) can be weakened to requiring the set to be homeomorphic to a convex set. Further generalizations can be derived using Lefschetz fixed-point theorem~\cite{lefschetz1926intersections}.
\end{remark}

Since Brouwer fixed-point theorem is not constructive, Theorem~\ref{thm:existence} suffers from the same drawback: it guarantees the existence of a solution to Problem~\ref{prob:concretization} but does not provide a way to compute it. In what follows, we consider different cases where Problem~\ref{prob:concretization} can be solved efficiently.

\subsection{Systems with shared input dependency}\label{sec:sharedInputMatrix}

In this section, we consider dynamics and over-approximations that share the same input dependency% <-- This is important to avoid a black line.
\begin{subequations}\label{eq:constantB}
\begin{align}
f(x,u)&=f_\x(x)+f_{\x\u}(x,u),\\
\hat{f}(x,u)&=\hat{f}_\x(x)+f_{\x\u}(x,u).
\end{align}
\end{subequations}
An important special case is when the shared part of the dynamics is linear, and the over-approximation is affine, i.e.,
$$
f(x,u)=f_\x(x)+Bu,\quad \hat{f}(x,u)=Ax+a+Bu,
$$
which makes the over-approximation easier to control.

In the case \eqref{eq:constantB}, the error between the two systems is $e(x,u)=f_\x(x)-\hat{f}_\x(x)$, which is independent of $u$. Consequently, Problem~\ref{prob:concretization} simplifies to an explicit policy evaluation:
\begin{align}\label{eq:concretization:constantB}
u=\pi(x,f_\x(x)-\hat{f}_\x(x)),
\end{align}
which has a solution for any policy $\pi:\X\times \E\rightarrow \U$ (even if $\pi$ is discontinuous). This is summarized in the following theorem.

\begin{theorem}\label{thm:constant}
Let $f,\hat{f}:\X\times\U\rightarrow\R^{n_x}$ and $\pi:\X\times\E\rightarrow\U$ be three functions, and assume that $\E$ satisfies \eqref{eq:errorInclusion}. If $f$ and $\hat{f}$ are given by \eqref{eq:constantB}, then for any $x\in\X$, Problem~\ref{prob:concretization} has a unique solution given by \eqref{eq:concretization:constantB}.
\end{theorem}

\subsection{Input affine systems}

In this section, we consider input-affine dynamics and over-approximations, i.e.,% <-- This might be useful to avoid a blanc line
\begin{subequations}\label{eq:affine}
\begin{align}\label{eq:inputAffine}
f(x,u)=f_\x(x)+f_\u(x)u,\ \hat{f}(x,u)=\hat{f}_\x(x)+\hat{f}_\u(x)u.
\end{align}
In addition, we consider informed policies that are error-affine, i.e.,
\begin{align}\label{eq:errorAffine}
\pi(x,e)=\pi_\x(x)+\pi_\e(x)e.
\end{align}
\end{subequations}
Under these assumptions, Problem~\ref{prob:concretization} reduces to:
\begin{align}\label{eq:concretization:affine}
\text{Find}\ u\in\U \text{ such that }
u=\pi_\x(x)&\\
+\pi_\e(x) \Big(f_\x(x) - \hat{f}_\x(x) +& \big(f_\u(x)-\hat{f}_\u(x)\big)u\Big). \notag
\end{align}
Since the constraint in \eqref{eq:concretization:affine} is linear in $u$, \eqref{eq:concretization:affine} is a convex program if $\U$ is convex, and a linear program if $\U$ is a polytope. This is summarized in the following theorem.

\begin{theorem}
Let $f,\hat{f}:\X\times\U\rightarrow\R^{n_x}$ and $\pi:\X\times\E\rightarrow\U$ be three functions, and assume that $\E$ satisfies \eqref{eq:errorInclusion}. If $f,\hat{f}$ and $\pi$ are given by \eqref{eq:affine}, and if the set $\U$ is non-empty, compact and convex, then for any $x\in\X$, Problem~\ref{prob:concretization} has a solution that can be obtained by solving \eqref{eq:concretization:affine}, which is a convex program. If in addition, $\U$ is a polytope, then \eqref{eq:concretization:affine} is a linear program. 
\end{theorem}
\begin{proof}
The existence of a solution to \eqref{eq:concretization:affine} follows from Theorem~\ref{thm:existence}. The convexity and linearity of the optimization program are straightforward.
\end{proof}

\subsection{Nonlinear systems}
In this section, we consider an arbitrary nonlinear system $f$ and an arbitrary over-approximation $\hat{f}$.

Banach fixed-point theorem \cite{banach1922operations} establishes an iterative method to solve a fixed-point equation. This theorem is recalled here.

\begin{lemma}[Banach fixed-point theorem]\label{thm:banach}
Let $X$ be a non-empty and complete space, and let $g:X\rightarrow X$ be a function. If $g$ is a contraction, i.e., $L_g<1$, then it admits a unique fixed point, and the sequence $x^{(k+1)}=g(x^{(k)})$ converges to that fixed point for any initial $x^{(0)}\in X$.
\end{lemma}

The following result is a direct application of Banach fixed-point theorem to the function $\mathcal{F}_x$ defined by \eqref{eq:FPoperator}.

\begin{theorem}\label{thm:concretization:nonlinear}
Let $f,\hat{f}:\X\times\U\rightarrow\R^{n_x}$ and $\pi:\X\times\E\rightarrow\U$ be three functions, and assume that $\E$ satisfies \eqref{eq:errorInclusion}. If the set $\U$ is non-empty and complete, and if for a given $x\in\X$, the function $\mathcal{F}_x$ defined by \eqref{eq:FPoperator} is a contraction, i.e., $L_{\mathcal{F}_x}<1$, then Problem~\ref{prob:concretization} has a unique solution and the sequence
$$
u^{(k+1)}=\mathcal{F}_x(u^{(k)})
$$
converges to that solution for any initial $u^{(0)}\in\U$.
\end{theorem}

From \eqref{eq:FPoperator}, two sufficient conditions for the contractivity of $\mathcal{F}_x$ are
\begin{equation}\label{eq:lipschitz}
L_{\pi(x,\cdot)}<\dfrac{1}{L_{(f-\hat{f})(x,\cdot)}} \text{ or } L_{\pi(x,\cdot)}<\dfrac{1}{L_{f(x,\cdot)}+L_{\hat{f}(x,\cdot)}}.
\end{equation}
The second condition is more conservative, but the Lipschitz constants involved may be easier to obtain. These conditions can be interpreted as small gain constraints on the policy. Finally, note that as soon as the denominators are finite, a policy $\pi(x,e)$ that is constant with respect to its second argument satisfies the contraction requirement. It follows that these conditions are not more conservative than considering uninformed policies that do not depend on the error.

\begin{remark}[Memoryful policies]\label{rem:memory}
The content of Section~\ref{sec:IP} generalizes to policies with a state and error memory, i.e.,
$$
u_t=\pi_t(x_{0:t};\bar{e}_{0:t}).
$$
Indeed, at time $t$, the states $(x_\tau)_{\tau=0}^t$ and the errors $(\bar{e}_\tau)_{\tau=0}^{t-1}$ are known since $\bar{e}_\tau=f(x_\tau,u_\tau)-\hat{f}(x_\tau,u_\tau)$. Consequently, the concretization problem corresponding to Problem~\ref{prob:concretization} is: Find $u_t\in\U$ such that
$$
u_t=\pi_t(x_{0:t};\bar{e}_{0:t-1},f(x_t,u_t)-\hat{f}(x_t,u_t)),
$$
which is a fixed-point equation.
\end{remark}

\begin{remark}[Nondeterministic dynamics]
The content of Section~\ref{sec:IP} generalizes to nonlinear systems with additive disturbances, i.e., $f:\X\times\U\times\mathcal{W}\rightarrow\R^{n_x}$:
$$
f(x,u,w)= f_{\x\u}(x,u)+w,
$$
where $\mathcal{W}\subseteq\R^{n_x}$ is a set of unknown disturbances. To this end, let us consider an approximation that includes additive disturbance as well, i.e., $\hat{f}:\X\times\U\times\mathcal{W}\rightarrow\R^{n_x}$ such that
$$
\hat{f}(x,u,w)= \hat{f}_{\x\u}(x,u)+w.
$$
It leads to an error function
$$
e(x,u)\coloneqq f(x,u,w)-\hat{f}(x,u,w) = f_{\x\u}(x,u) - \hat{f}_{\x\u}(x,u),
$$
which is independent of the disturbance $w$. Therefore, an informed policy $\pi:\X\times\E\rightarrow\U$ (which depends on $\bar{e}$, but does not depend on $w$) can be concretized by solving Problem~\ref{prob:concretization}. This leads to the over-approximation relation:
$$
f(x,u,\mathcal{W})\subseteq \{ \hat{f}(x,\pi(x,\bar{e}),w)+\bar{e} \mid \bar{e}\in\E, w\in\mathcal{W} \}.
$$
\end{remark}

Section~\ref{sec:IP} contains our main contributions. For self-containment, we briefly review in Section~\ref{sec:computingOA} one method to compute over-approximations, and in Section~\ref{sec:synthesizingIP} one method to synthesize informed policies. Section~\ref{sec:IP} is independent of the specific methods reviewed below.

\section{COMPUTING LINEAR OVER-APPROXIMATIONS} \label{sec:computingOA}
We consider a linear over-approximation
$$
\hat{f}(x,u)=Ax+Bu,
$$
with error set $\E=\bigtimes_{i=1}^{n_x} [\underline{\epsilon}_i,\bar{\epsilon}_i]\subseteq\R^{n_x}$ that is an axis-aligned box. Finding the tightest $\E$ such that \eqref{eq:errorInclusion} holds corresponds to solving
\begin{align}
\min_{A,B,\underline{\epsilon},\bar{\epsilon}}\ \sum_i\bar{\epsilon}_i-\underline{\epsilon}_i \text{ such that}& \label{eq:opti:over-approx} \\
\text{for all } (x,u)\in\X\times\U:\ \underline{\epsilon}\leq& f(x,u)-Ax-Bu\leq\bar{\epsilon}.\notag
\end{align}
This problem is infinite-dimensional because the set $\X\times\U$ contains infinitely many elements. However, different techniques exist to find feasible solutions of small cost. If the dynamics is polynomial and the sets $\X$ and $\U$ are semi-algebraic, sum-of-squares optimization can be used (see \cite[Section~4.1]{aspeel2024simulation}). In the general case, a feasible solution can be found if (an upper bound on) the Lipschitz constant of $f$ is known \cite{malherbe2017global}, \cite[Section~IV-A]{balim2023koopman}. Indeed, the constraint in \eqref{eq:opti:over-approx} can be enforced on a finite number of samples $\mathcal{D}=\{(x^j,u^j)\}_{j\in J}\subseteq\X\times\U$ instead of for all $(x,u)\in\X\times\U$, leading to a finite-dimensional linear program. The corresponding set $\E$ can then be enlarged to guarantee that the constraint holds everywhere on $\X\times\U$ and not only on $\mathcal{D}$. To this end, consider the dispersion of the dataset $\mathcal{D}$ in $\X\times\U$ \cite[Section~5.2.3]{lavalle2006planning}:
$$
\delta\coloneqq \sup_{(x,u)\in\X\times\U}\min_{(x^j,u^j)\in\mathcal{D}} \|(x,u)-(x^j,u^j)\|_\infty.
$$
It corresponds to how far a point in $\X\times\U$ can be from a point in $\mathcal{D}$. Then, the set $\E$ is given by
$$
\underline{\epsilon}_i= \underline{\epsilon}_i^\mathcal{D}-L_{e_i}\delta \text{ and } \bar{\epsilon}_i = \bar{\epsilon}_i^\mathcal{D}+L_{e_i}\delta,
$$
where $\underline{\epsilon}^\mathcal{D}$ and $\bar{\epsilon}^\mathcal{D}$ are the optimal solution of the finite-dimensional linear program, and $L_{e_i}$ is the Lipschitz constant of the $i$-th component $e_i$ of the error function defined in \eqref{eq:error}, with respect to the infinity norm. If hard to compute, this Lipschitz constant can be over-approximated by $L_{e_i}\leq L_{f_i}+\lVert \begin{bmatrix} A & B \end{bmatrix}_{i,:} \rVert_{\infty\rightarrow\infty}$, where $\|\cdot\|_{\infty\rightarrow\infty}$ denotes the induced infinity norm.

\section{SYNTHESIZING INFORMED POLICIES} \label{sec:synthesizingIP}

In this section, we show how synthesizing an informed policy reduces to the problem of synthesizing an (uninformed) state-feedback policy for an augmented system. This augmentation method was proposed in~\cite{liu2021value, liu2024quantifying} in the context of control with preview information. For the over-approximation $x_{t+1}=F(x_t,u_t,\bar{e}_t)$, let us consider the augmented dynamics
\begin{align} \label{eq:augmentedSystem}
\begin{bmatrix}x_{t+1}\\ \bar{e}_{t+1}\end{bmatrix} =\begin{bmatrix}F(x_t,u_t,\bar{e}_t)\\ \tilde{e}_t\end{bmatrix},
\end{align}
whose state is $z_t=(x_t,\bar{e}_t)\in\R^{2n_x}$, input is $u_t$ and perturbation is $\tilde{e}_t\coloneqq \bar{e}_{t+1}\in \E$. Any state-feedback policy $u_t=\pi(z_t)$ for system \eqref{eq:augmentedSystem} directly gives an informed policy for $F$. Note that if $F$ is linear or piece-wise linear, then also is the augmented system~\eqref{eq:augmentedSystem}.

\subsection{System level synthesis}

The System Level Synthesis (SLS) framework parameterizes all linear causal controllers through their closed-loop mappings from disturbances to states and inputs \cite{anderson2019system}. This representation renders common performance and robustness constraints convex in the corresponding system responses, allowing one to synthesize a policy via a convex optimization program.
In this section, we show how SLS can be applied to the augmented system \eqref{eq:augmentedSystem} to synthesize an informed policy. To this end, consider an augmented linear over-approximation:
$$
\begin{bmatrix}
x_{t+1}\\ \bar{e}_{t+1}
\end{bmatrix} =\begin{bmatrix}A & I \\0 & 0\end{bmatrix}\begin{bmatrix}
x_{t}\\ \bar{e}_{t}
\end{bmatrix} +
\begin{bmatrix}
B \\ 0
\end{bmatrix} u_t +
\begin{bmatrix}
0\\I
\end{bmatrix}\tilde{e}_t,
$$
which can be written
\begin{equation}\label{eq:sls:augmentedSystem}
z_{t+1}=\tilde{A}z_t+\tilde{B}u_t+\tilde{D}\tilde{e}_t.
\end{equation}
We consider a linear policy with a memory:
\begin{equation}\label{eq:linearPolicyMemory}
u_t = \sum_{\tau=0}^t K_{(t,\tau)}^\x x_\tau + K_{(t,\tau)}^\e \bar{e}_\tau.
\end{equation}
Writing $K^\z_{(t,\tau)}=\begin{bmatrix}K_{(t,\tau)}^\x & K_{(t,\tau)}^\e\end{bmatrix}$, the control law can be written $u_t=\sum_{\tau\leq t}K^\z_{(t,\tau)}z_\tau$. Let us remind that informed policies with a memory can be concretized according to Remark~\ref{rem:memory}.

Let us introduce the vectors $\tilde{\mathbf{e}}=( z_0, \tilde{e}_{0:T-1})$, $\mathbf{z}=z_{0:T}$, $\mathbf{u}=u_{0:T}$ and the matrices $\mathcal{A}\coloneqq I_{T+1}\otimes \tilde{A}$, $\mathcal{B}\coloneqq I_{T+1}\otimes \tilde{B}$, and $\mathcal{D}\coloneqq \text{blkdiag}(I_{2n_x}, I_{T}\otimes \tilde{D})$, where $\otimes$ and $\text{blkdiag}$ indicate a Kronecker product, and a block-diagonal concatenation, respectively. In addition, let $\mathcal{Z}\in\R^{(T+1)2n_x\times(T+1)2n_x}$ be the block-downshift operator, i.e., the $(2n_x,2n_x)$-block-lower-triangular matrix with identities on its first block subdiagonal and zeros elsewhere. Finally, consider the following matrix:
\begin{align*}
K^\z=\begin{bmatrix}
K_{(0,0)}^\z & & & \\
K_{(1,0)}^\z & K_{(1,1)}^\z & & \\
\vdots & & \ddots & \\
K_{(T,0)}^\z & K_{(T,1)}^\z & \cdots & K_{(T,T)}^\z
\end{bmatrix}.
\end{align*}
We use a similar notation to denote the blocks of any block-lower-triangular matrix. Specifically, for a block-lower-triangular matrix $M \in \mathbb{R}^{(T+1)m \times (T+1)n}$ with block size $(m,n)$, we denote by
\begin{equation*}
M_{(t,\tau)} \coloneqq M_{tm+1:(t+1)m,\;\tau n+1:(\tau+1)n}
\end{equation*}
its $(t,\tau)$-block, for all $t,\tau \in \{0,\dots,T\}$.

System level synthesis parameterizes the controller matrix $K^\z$ through the following system responses $\Phi^\z$ and $\Phi^\u$:
\begin{equation}\label{eq:sls:response}
\mathbf{z}=\Phi^\z\mathcal{D}\tilde{\mathbf{e}}, \quad \mathbf{u}=\Phi^\u\mathcal{D}\tilde{\mathbf{e}}.
\end{equation}
The following theorem is the core result of finite-horizon SLS. It provides conditions under which all linear causal controllers can be parameterized by their system responses $\Phi^\z$ and $\Phi^\u$.

\begin{theorem}[{\cite[Theorem~2.1]{anderson2019system}}]
There is a $K^\z$ for which system \eqref{eq:sls:augmentedSystem} gives the system response \eqref{eq:sls:response} if and only if the matrices $\Phi^\z$ and $\Phi^\u$ are $(2n_x,2n_x)$- and $(n_u,2n_x)$-block-lower-triangular, respectively, and  
\begin{equation}\label{eq:sls:slc}
\begin{bmatrix}
I-\mathcal{Z}\mathcal{A} & -\mathcal{ZB}
\end{bmatrix}
\begin{bmatrix}
\Phi^\z\\ \Phi^\u
\end{bmatrix}=I.
\end{equation}
In addition, $K^\z=\Phi^\u(\Phi^\z)^{-1}$ achieves the system response \eqref{eq:sls:response}.
\end{theorem}
Note that the constraint \eqref{eq:sls:slc} is affine in $\Phi^\z$ and $\Phi^\u$.

\subsection{Contractivity condition}

When the true dynamics $f$ is nonlinear and not input-affine, concretizing an informed policy relies on the iterative procedure described in Theorem~\ref{thm:concretization:nonlinear}. To ensure convergence of this iteration, the policy must satisfy the contraction condition \eqref{eq:lipschitz}. When the policy has the form \eqref{eq:linearPolicyMemory}, condition \eqref{eq:lipschitz} simplifies to the small-gain constraint $\|K^\e_{(t,t)}\|< \gamma$ for all $t$, where $\gamma\geq0$ is a constant determined by the Lipschitz constants in the right-hand sides of \eqref{eq:lipschitz}. The next theorem shows how this small-gain constraint can be enforced convexly in terms of the system-response variable $\Phi^\u$ provided by the SLS formulation.
\begin{theorem}\label{thm:sls:contractivity}
If $\Phi^\z$ and $\Phi^\u$ satisfy \eqref{eq:sls:slc}, and $\begin{bmatrix}K^\x & K^\e \end{bmatrix}=\Phi^\u(\Phi^\z)^{-1}$, then $K^\e_{(t,t)}=(\Phi^\u_{(t,t)})_{:,n_x+1:2n_x}$. In addition, for any $\gamma\geq0: \|K^\e_{(t,t)}\|< \gamma$ if and only if $\| (\Phi^\u_{(t,t)})_{:,n_x+1:2n_x} \|< \gamma$.
\end{theorem}
\begin{proof}
Equation~\eqref{eq:sls:slc} can be written
$$
\Phi^\z=I+\mathcal{Z}(\mathcal{A}\Phi^\z+\mathcal{B}\Phi^\u).
$$
Since all the matrices involved are block-lower-triangular, and the matrix $\mathcal{Z}$ is strictly block-lower-triangular, it follows that $\Phi^\z$ is a block-lower-triangular matrix with identities on its diagonal. Consequently, its inverse $(\Phi^\z)^{-1}$ also is a block-lower-triangular matrix with identities on its diagonal. It follows that $K^\z_{(t,t)}=(\Phi^\u(\Phi^\z)^{-1})_{(t,t)}=\Phi^\u_{(t,t)}$. The rest of the proof follows directly.
\end{proof}

\section{EXPERIMENTS}

To illustrate our method, we consider a nonlinear system of the form \eqref{eq:nonlinearSystem} over a finite horizon $t=0,\dots,T$, with a known initial state. We synthesize a policy of the form~\eqref{eq:linearPolicyMemory} such that
\begin{equation}\label{eq:reachability}
(x_t,u_t)\in\X\times\U
\end{equation}
for each $t$. We want to find the largest $\alpha\in\R$ such that $x_{1,T}\geq \alpha$, i.e., the first component of the state is as large as possible at the end of the horizon. To this end, we compute a linear over-approximation using the method described in Section~\ref{sec:computingOA} with a uniform grid $\mathcal{D}$ that has a dispersion $\delta=0.03$ in $\X\times\U$. The examples consider a continuous-time system $f_\text{cont}$ discretized using forward Euler scheme with step $\kappa=0.1$, i.e.,
\begin{equation}\label{eq:Euler}
f(x,u)=x+\kappa f_\text{cont}(x,u).
\end{equation}
The Lipschitz constant of each component of the error is over-approximated using the Lipschitz constants of the continuous-time dynamics. From \eqref{eq:Euler} and $e(x,u)=f(x,u)-Ax-Bu$, we have
$$
L_{e_i}\leq \kappa L_{f_{\text{cont},i}} +\lVert\begin{bmatrix}A-I & B\end{bmatrix}_{i,:}\rVert_{\infty\rightarrow\infty}.
$$

Finally, we synthesize an informed policy using the method described in Section~\ref{sec:synthesizingIP}. In terms of the augmented system \eqref{eq:sls:augmentedSystem}, constraints \eqref{eq:reachability} and $x_{1,T}\geq \alpha$ are enforced for all $\bar{e}_0,\dots, \bar{e}_{T-1}\in \E$. When $\X_t$, $\U$, $\E$ and $\X_0$ are polyhedra, constraint \eqref{eq:reachability} can be handled linearly using system level synthesis and Farkas' lemma. See \cite[Section~III-B]{aspeel2023low} for details. For comparison, an uninformed policy is synthesized via system-level synthesis on the unaugmented linear over-approximation.

A \texttt{Julia} code that implements our method and generates the figures is available at {\tt\small\url{https://github.com/aaspeel/InformedOverApproximations}}.

\subsection{Input-affine dynamics}
In this first experiment, we consider the following dynamics
\begin{subequations}\label{eq:inputAffineExperiment}
\begin{align}
\dot{x}_1&=x_2+0.5 x_1^2 \\
\dot{x}_2&=0.75x_1^2+2x_2^3+\cos(x_2)u.
\end{align}
\end{subequations}
The state and input sets are $\X=[0,1]\times[-0.5,0.5]$ and $\U=[-1,1]$. The discrete time horizon is $T=30$ time steps. The Lipschitz constants used to compute the over-approximation are $L_{f_{\text{cont},1}}=2$ and $L_{f_{\text{cont},2}}=5$.

Since this dynamics is input-affine and the set $\U$ is a polytope, the informed policy is concretized by solving the linear program \eqref{eq:concretization:affine}.

The results are shown in Fig.~\ref{fig:inputAffineExperiment}, where the informed policy allows to guarantee that the terminal state will satisfy $x_{1,T}\geq \alpha_\text{informed}=0.81$, while an uninformed policy can only guarantee $x_{1,T}\geq \alpha_\text{uninformed}=0.62$.

\begin{figure}
\centering
\includegraphics[width=8cm]{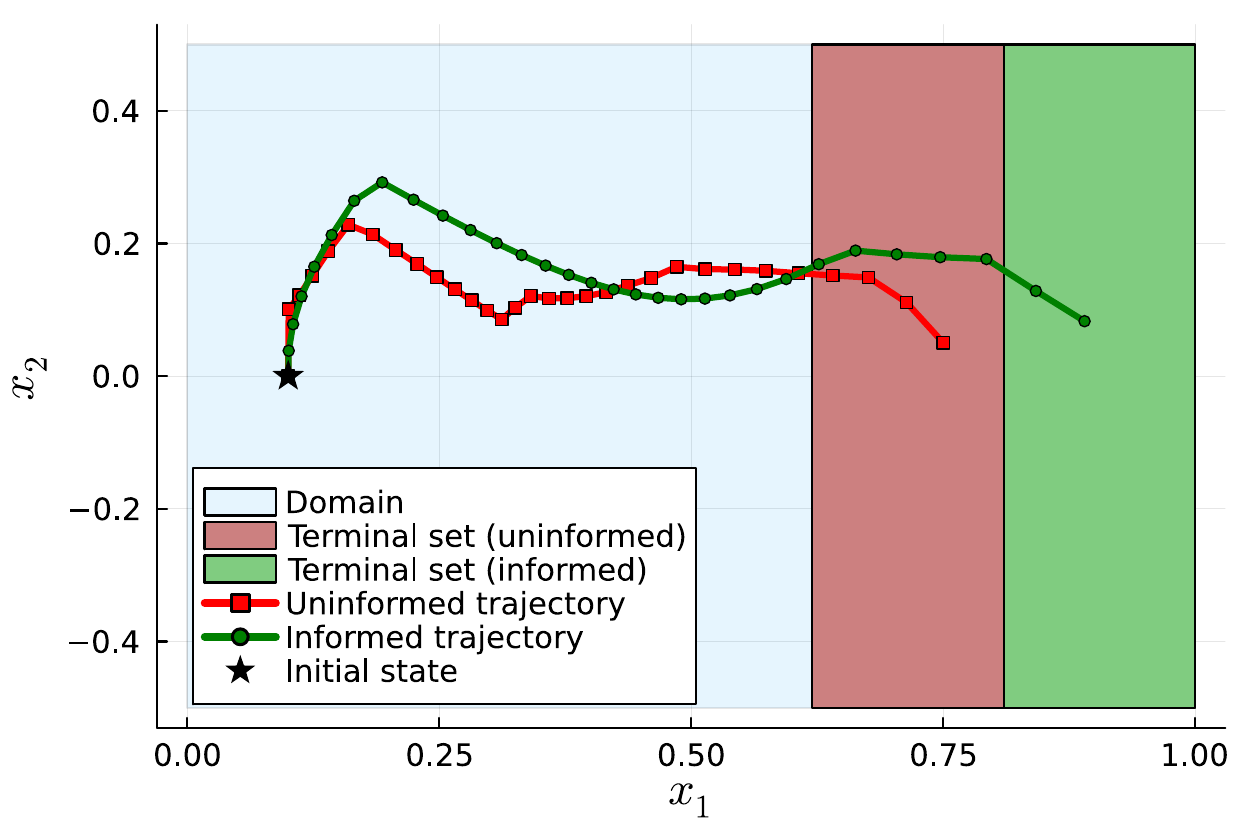}
\caption{\small Trajectories obtained with informed and uninformed policies for the dynamics \eqref{eq:inputAffineExperiment}. The informed policy guarantees to be further on the right at the end of the horizon.}
\label{fig:inputAffineExperiment}
\end{figure}

\subsection{Nonlinear dynamics}
As a second example, we consider the dynamics
\begin{subequations}\label{eq:nonlinearExperiment}
\begin{align}
\dot{x}_1&=x_2 \\
\dot{x}_2&=-\sin(x_1)+0.5x_2^2+2\sin(u).
\end{align}
\end{subequations}
The state and input sets are $\X=[-3.14, 3.14+3.14/12]\times[-2,2]$ and $\U=[-3.14/2,3.14/2]$, respectively. The discrete time horizon is $T=35$ time steps. The Lipschitz constants used to compute the over-approximation are $L_{f_{\text{cont},1}}=1$ and $L_{f_{\text{cont},2}}=5$.

Since the dynamics is not input-affine, the informed policy is concretized using the fixed-point iteration described in Theorem~\ref{thm:concretization:nonlinear}. The second contractivity condition in \eqref{eq:lipschitz} is enforced with $L_{f(x,\cdot)}=2$ and $L_{\hat{f}(x,\cdot)}=\|B\|_{\infty\rightarrow\infty}$. This constraint is incorporated into the system-level-synthesis program according to Theorem~\ref{thm:sls:contractivity}, which is a linear constraint. The strict inequality is replaced by a non-strict one with tolerance $10^{-5}$.

The results are reported in Fig.~\ref{fig:nonlinearExperiment}, where $\alpha_\text{informed}=3.1$ and $\alpha_\text{uninformed}=2.7$, showing the reduction in conservatism obtained from using an informed policy.

\begin{figure}
\centering
\includegraphics[width=8cm]{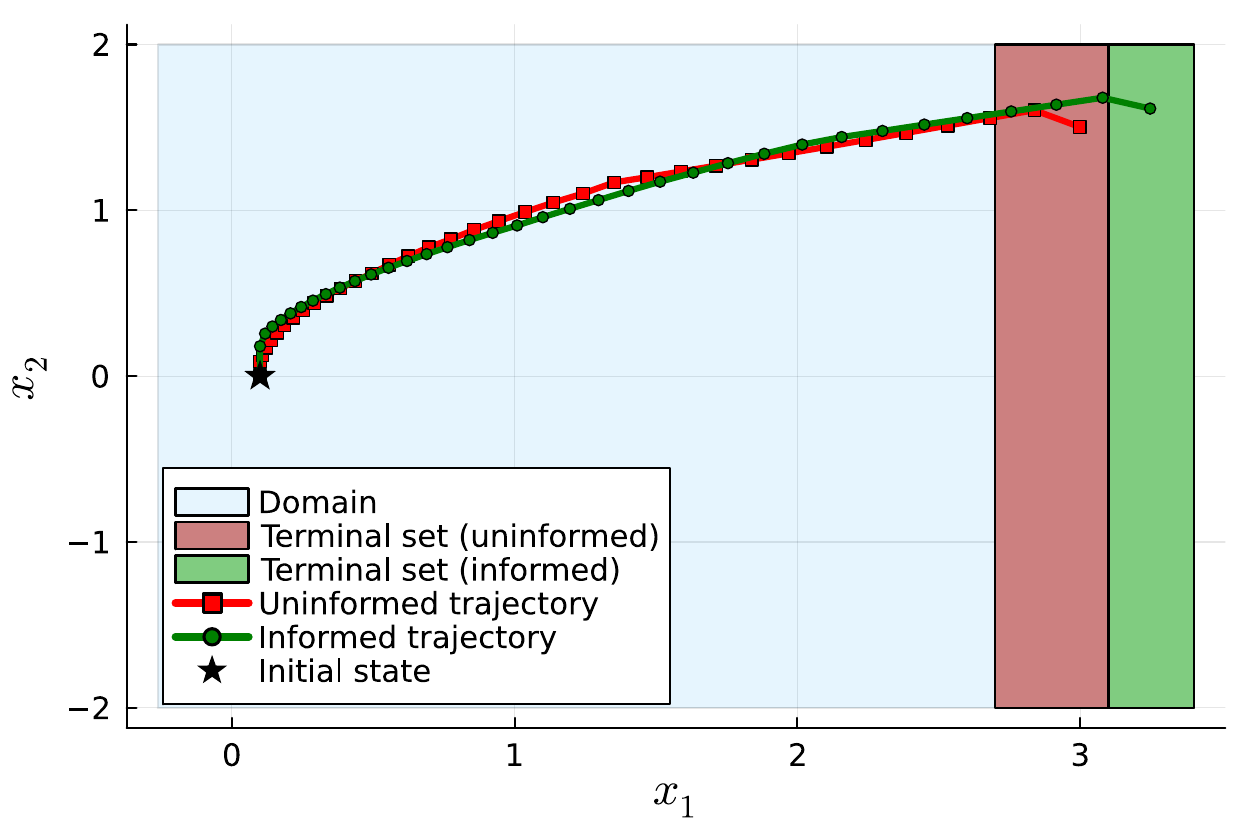}
\caption{\small Trajectories obtained with informed and uninformed policies for the dynamics \eqref{eq:nonlinearExperiment}. The informed policy guarantees to be further on the right at the end of the horizon.}
\label{fig:nonlinearExperiment}
\end{figure}

\section{CONCLUSION}
We proposed \emph{informed policies}, which exploit over-approximation errors as input-dependent preview information rather than treating them as adversarial disturbances. By formulating concretization as a fixed-point problem, we established general existence guarantees and provided efficient solution methods: closed-form or convex programs for input-affine systems, and a contraction-based iterative scheme for nonlinear systems. Numerical examples illustrate that informed policies significantly reduce conservatism compared to standard approaches.

Future work includes extending the framework to systems with disturbances, incorporating output feedback, and addressing mismatched dimensions between $f$ and $\hat{f}$, as well as multi-step preview information.

\bibliographystyle{alpha}
\bibliography{references}

\end{document}